\documentclass[reqno]{amsart}

\usepackage{amssymb} 
\usepackage{amsmath, amsthm} 

\usepackage[top=3cm, bottom=3cm, left=3cm, right=3cm]{geometry}

\usepackage{color}

\usepackage{enumitem} 
\usepackage{multirow} 
\usepackage{cleveref} 
\usepackage{comment}

\usepackage{tikz}
\usetikzlibrary{intersections}
\usetikzlibrary{decorations.markings}
\usetikzlibrary{arrows}
\usetikzlibrary{positioning}

\usepackage[T1]{fontenc}
\DeclareMathAlphabet{\mathpzc}{OT1}{pzc}{m}{it}
\usepackage{yfonts}
\usepackage[sans]{dsfont}
\usepackage{txfonts}
\usepackage[mathscr]{euscript}
\usepackage{bbm}

\newtheorem{thm}{Theorem}
\newtheorem{cor}[thm]{Corollary}
\newtheorem{lem}[thm]{Lemma}

\theoremstyle{definition}		
\newtheorem{defn}[thm]{Definition}	

\AtBeginEnvironment{lem}{\crefalias{thm}{lem}}

\newcommand{\grp}[1]{\mathfrak{#1}}	
\newcommand{\proj}[1]{\textgoth{#1}}	

\newcommand{\pnt}[1]{\mathpzc{#1}}	
\newcommand{\lne}[1]{\mathpzc{#1}}	

\newcommand{\nat}{\mathbb N}			

\newcommand{\cyc}[1]{\langle #1 \rangle}	
\DeclareMathOperator{\norm}{\lhd\,}	
\DeclareMathOperator{\xs}{\ltimes}		
\DeclareMathOperator{\Ker}{Ker}			
\DeclareMathOperator{\Aut}{Aut}			
\DeclareMathOperator{\Syl}{Syl}			
\newcommand{\id}{\mathpzc{id}}			

\DeclareMathOperator{\PG}{PG}			
\newcommand{\nearf}[1]{\mathds{#1}}		
\newcommand{\N}{\nearf{N}}			
\newcommand{\NP}{\proj{N}}			
\DeclareMathOperator{\Inc}{I}			

\newcommand{\ff}[1]{{\mathbb F}_{#1}}		
\newcommand{\ffs}[1]{{\mathbb F}_{#1}^\star}	
\newcommand{\ffx}[1]{\ff{#1}[X]}		
\newcommand{\gen}{\mathfrak{z}}			
\newcommand{\gn}{\mathfrak{g}}			

\DeclareMathOperator{\Tr}{Tr}
\DeclareMathOperator{\Nm}{N}

\def\nonsquare{\boxslash}		
\def\nm{\star}				
\DeclareMathOperator{\ord}{ord}		
\newcommand\ordp[1]{\ord_p(#1)}

\newcommand{\U}{\mathcal U}
\newcommand{\V}{\mathcal V}

\begin{document}
\title[Unitals in nearfield planes]{A classification of unitals in
nearfield planes with maximal automorphism group}

\author[R.S. Coulter]{Robert S. Coulter}
\author[A.M.W. Hui]{Alice Man Wa Hui}
\author[R. Weaver]{Randon Weaver}

\address[R.S. Coulter and R. Weaver]{Department of Mathematical Sciences, University of Delaware, Newark, DE, 19716, United States of America.}

\address[A.M.W. Hui]{School of Mathematics and Statistics, Clemson University, Clemson, SC, 29634, United States of America.}

\begin{abstract}
We classify the parabolic unitals in regular nearfield planes of odd order
$q^2$ whose linear collineation group has the maximal size of $q^3-q$.
We also establish a number of more general results concerning parabolic unitals
in regular nearfield planes under weaker assumptions.
\end{abstract}

\maketitle

\section{Introduction}

Set $p$ to be a prime and $q$ a power of $p$. We use
$\ff{q}$ to denote the finite field of order $q$ and $\ffs{q}$ to denote the
nonzero elements.
The squares and non-squares of $\ff{q}$ are denoted by $\square_q$ and
$\nonsquare_q$, respectively.

A {\it unital} of order $m$ is a $2$-$(m^3 + 1,m + 1,1)$ design.
In other words, it has $m^3+1$ points, $m+1$ points in each block, and the property that there is a unique block through any two points.
There are $m^2$ blocks on each point.
If a unital of order $m$ is embedded in a projective plane of order $m^2$,
then each line must intersect the unital in either $m+1$ points or $1$ point.
Lines of the former type are called {\it secants}, and lines of the latter type are called {\it tangents}.
There are $m^2$ secants and 1 tangent through each point of the unital, while
for any point not in the unital, there are $m^2-m$ secants and $m + 1$ tangents.
If a line of the plane is chosen as the line at infinity,
then the unital is called {\it hyperbolic} or {\it parabolic} depending on
whether the line at infinity intersects the unital in a secant or a tangent.

There are a number of known constructions of unitals embedded in
projective planes of square order, see the monograph \cite{bBE-2008-uipp}
of Barwick and Ebert.
An example is the {\it classical unital} given by the nondegenerate Hermitian
curve in the Desarguesian plane $\PG(2,q^2)$ over the finite field $\ff{q^2}$. One of its canonical forms is
\begin{equation}\label{classicalU}
\{(x,y,z)\in \PG(2,q^2)\,:\, x^{q+1} - y^q z - z^q y =0 \}.
\end{equation}
A generalization of the classical unital is the {\em orthogonal-Buekenhout-Metz
unital} (or orthogonal BM unital, for short) in a translation plane of dimension two over its kernel, which is obtained from
an elliptic cone under the Bruck-Bose representation of the plane, see
Buekenhout \cite{B-1976-eouif}, Theorem 3.
Orthogonal-BM unitals in $\PG(2,q^2)$ were studied and enumerated by Baker
and Ebert \cite{BE-1992-obmuo} and Ebert \cite{E-1992-obmuo}.

In this article we will be principally concerned with unitals in nearfield
planes, which are particularly special examples of translation planes.
(More information on nearfields and their planes is given below.)
Our work is preceded by several articles on this topic. In particular,
Wantz \cite{W-2008-uitrn} studied a class of orthogonal-BM unitals
in the regular nearfield planes of order $q^2$. Let $\NP$ denote a regular
nearfield plane of order $q^2$. Wantz showed that the set
\begin{equation}\label{WantzU}
\{(x,a x^2 + b x^{q+1} + t \epsilon ,1)
\,:\, x \in \ff{q^2}, t\in \ff{q} \}\cup \{(0,1,0)\},
\end{equation}
is a unital in $\NP$ if and only if $a,b \in \ff{q^2}$ and
$b^2-a^2 \in \square_q$.
Baker and Ebert enumerated the distinct orthogonal-BM unitals in regular
nearfield planes of order $q^2$ in \cite{BEW-2010-eoobu}.

In this article we are especially interested in the case where $q$ divides
the order of $\Aut(U)$ for a unital $U$ in $\NP$.
Our main result is a full classification of such unitals in the largest
possible automorphism group case. In fact, we prove
all are isomorphic to the unitals of Wantz, see \Cref{TheWantzTheorem}.
We also prove several more general results of independent interest while
progressing towards this goal, see Theorems \ref{centralcollineationlemma},
\ref{restrictionthm}, and \ref{restrictionthm1y0}, in particular.

The article is structured as follows.
We first recall some results on nearfields, nearfield planes, and their
collineation groups in Sections 2 and 3. In Section 4, we prove
\Cref{centralcollineationlemma}, which is a general statement concerning central collineations that fix a unital.
In Section 5 we study the structure of the collineation groups that fix a
unital in a nearfield plane, under the assumption that the order of the
automorphism group is divisible by $q$. Then, in Sections 6 and 7, we examine how
these collineation groups dictate the geometry of the corresponding unitals,
and in particular the absence of certain O'Nan configurations. In Section 9,
assuming that the collineation group of the unital has order $q(q^2-1)$, we
prove that the fixed unital is isomorphic to one the Wantz unitals using
information about the all-ones polynomials presented in Section 8.

\section{Nearfields} \label{nearfields}

The algebraic structure we shall be principally concerned with is a nearfield.
\begin{defn}
A finite set $S$ with two binary operations, $+$ (addition) and $\nm$
(multiplication), is called a (right) {\em planar nearfield} if
\begin{itemize}
\item $\cyc{S,+}$ is an abelian group with identity $0$,
\item $\cyc{S^\star,\nm}$ is a group, and
\item the right distributive law holds.
\end{itemize}
\end{defn}
As with finite fields, the additive structure of a planar nearfield is
necessarily elementary abelian, so that every nearfield can be viewed as a
finite field with some new ``twisted" multiplication.
Dickson \cite{D-1905-doaga,D-1905-ofa} introduced the concept of a nearfield
in 1905 and constructed two types of planar nearfields, now known as the
regular and irregular nearfields.
In 1935, Zassenhaus \cite{Z-1935-uef} showed that these were the only ones.
For a modern treatment of this classification result, see Grundh\" ofer and
Hering \cite{GH-2017-fnca}.

There are exactly 7 irregular nearfields, and we refer the interested
reader to Dembowski \cite{bD-1968-fg}, page 231, for the specifics.
On the other hand, there are infinitely many regular nearfields.
To construct any one of them, let $q$ be a prime power
and $n$ a natural number such that the prime divisors of $n$ also divide $q-1$.
If $q \equiv 3 \pmod 4$, then we also require $n \not \equiv 0 \bmod 4$.
Let $\gen$ be a primitive element of $\ff{q^n}$ and let $C=\cyc{\gen^d}$,
the group generated by $\gen^d$.
Note that the order of $C$ is $o(C)=(q^n-1)/d$ and has index $d$ in
$\ffs{q^n}$.
It is an easy exercise to prove that a set of
coset representatives for $C$ in the multiplicative group
$\ffs{q^n}$ is $\{\gen_i\,:\, 0\le i\le d-1\}$, where
$\gen_i=\gen^{(q^i-1)/(q-1)}$.
Let $F$ be the Frobenius group of automorphisms of $\ff{q^n}$ over
$\ff{q}$; that is, $F = \cyc{\phi}$ where
$\phi:\ff{q^n}\rightarrow\ff{q^n}$ satisfies $\phi(x)=x^q$. It is clear that
$o(F)=n$.
We define the function $\alpha:\ffs{q^n} \rightarrow F$ by
$\alpha(y)= (x\mapsto x^{q^i})$ if and only if $y \in \gen_i C$.
Define a new multiplication on $\ff{q^n}$ by
\begin{equation*}
x\nm y =
\begin{cases}
x^{\alpha(y)} y &\text{ if $y\in\ffs{q^n}$,}\\
0 &\text{ if $y=0$.}
\end{cases}
\end{equation*}
Set $\N(n,q) = \cyc{\ff{q^n}, +,\nm}$, where $+$ is the field addition.
Then Dickson showed $\N(n,q)$ is a {\em regular} planar nearfield of order
$q^n$, with a right distributive law.

The construction of the multiplication of a regular nearfield is essentially
the construction of a semidirect product $C\xs F$. It is
easily seen that this group is necessarily metacyclic. This was proved by
Zassenhaus as part of his classification of nearfields \cite{Z-1935-uef},

In this article, we will be particularly interested in the regular nearfields
$\nearf N(2,q)$. In this situation, nearfield multiplication simplifies to
\begin{equation} \label{Nmult}
x\nm y =
\begin{cases}
xy &\text{ if $y\in\square_{q^2}$,}\\
x^q y &\text{ if $y\in\nonsquare_{q^2}$.}
\end{cases}
\end{equation}
With $x^{-1}$ and $\frac{1}{x}$ denoting the regular nearfield multiplicative inverse and of finite field inverse $x$ respectively, we get $x^{-1} = \frac{1}{x}$ if $x \in \square_{q^2}$ and $x^{-1} = \frac{1}{x^q}$ if $x \in \nonsquare_{q^2}.$
The multiplicative group $\cyc{\N(2,q)^\star,\star}$ is isomorphic to
$C_{(q^2-1)/2}\xs C_2$, where $C_d$ denotes the cyclic group of
order $d$. Let the representation of $C_2$ be $\{1,-1\}$.
Under this representation, the elements of $C_{(q^2-1)/2}\times \{1\}$
correspond to the squares of $\ffs{q^2}$ while the elements of
$C_{(q^2-1)/2}\times \{-1\}$ correspond to the nonsquares. Since the only
subgroups of $C_{(q^2-1)/2}\xs C_2$ are $C_d\times\{1\}$ or
$C_d\xs C_2$ for a divisor $d$ of $(q^2-1)/2$, we see that any
subgroup $S$ will consist entirely of squares of $\ffs{q^2}$, or split evenly
between squares and nonsquares of $\ffs{q^2}$.
Let $\gn$ be a fixed primitive element of $\ff{q^2}$.
In the former case, $S=\cyc{\gn^{(q^2-1)/d}}$.
In the latter case, where $S$ consists half of squares and half of nonsquares,
$S=S_{d/2}\cup S_{d/2} h$, where $h\in\nonsquare_{q^2}$
and $S_{d/2}=\cyc{\gn^{2(q^2-1)/d}}$ is of order $d/2$.
Moreover, $h\star h=h^{q+1}\in S$, forcing $h^{q+1}\in S_{d/2}$.
We summarize this discussion with the following lemma.
\begin{lem} \label{subgroups}
Let $\N=\nearf N(2,q)$ and let $\gn$ be a primitive element of $\ff{q^2}$.
If $S<\cyc{\N^\star,\star}$, then either
\begin{enumerate}[label=(\roman*)]
\item $S=\cyc{\gn^{(q^2-1)/d}}$ for some divisor $d$ of $(q^2-1)/2$, or
\item $S=S_{d/2}\cup S_{d/2} h$, where $S_{d/2}=\cyc{\gn^{2(q^2-1)/d}}$ for some divisor
$d$ of $(q^2-1)/2$ and $h\in\nonsquare_{q^2}$ satisfies $h^{q+1}\in S_{d/2}$.
\end{enumerate}
\end{lem}
Later we will find ourselves dealing with a surjective homomorphism mapping
one subgroup of $\cyc{\N(2,q)^\star,\star}$ to a cyclic subgroup of
$\cyc{\N(2,q)^\star,\star}$. Anticipating this, we prove the following.
\begin{lem} \label{CDrelation}
Let $G,H$ be subgroups of $\cyc{\N(2,q)^\star,\star}$ with $H$ cyclic, and let
$\sigma:G\rightarrow H$ be an onto homomorphism with $o(G)=r\, o(H)$ for some
integer $r$.
The following statements hold.
\begin{enumerate}[label=(\roman*)]
\item Suppose $G$ is cyclic. Then $H<G$ and there exists an integer
$1\le j < o(H)$ with $\gcd(j,o(H))=1$ for which $\sigma(x)=x^{r j}$ for all
$x\in G$.
\item Suppose $G$ is not cyclic, so that $G=S_l\cup S_l h$ with
$h^{q+1}\in S_l$.
Set $S=\sigma(S_l)$ and $K=\Ker(\sigma)\cap S_l$.
\begin{enumerate}[label=(\alph*)]
\item If $\sigma(S_l)=H$, then $r$ is even, $H<G$, $o(K)=r/2$, and there
exists an integer
$1\le j<o(H)$ with $\gcd(j,o(H))=1$ for which $\sigma(x)=x^{r j/2}$ for all
$x\in S_l$.

\item If $\sigma(S_l)\ne H$, then $o(H)$ is even, $o(K)=r$, and there
exists an integer
$1\le j<o(H)/2$ with $\gcd(j,o(H)/2)=1$ for which $\sigma(x)=x^{r j}$ for all
$x\in S_l$. Furthermore, $\sigma(h)\in H\setminus \sigma(S_l)$ and
$H=S_l\cup S_l\sigma(h)$.
\end{enumerate}
\end{enumerate}
\end{lem}
\begin{proof}
Since $G<\cyc{\N(2,q)^\star,\star}$, we know the form of $G$ from \Cref{subgroups}.

If $G$ is cyclic, then both $G$ and $H$ are necessarily subgroups of the
(cyclic) multiplicative group of $\ffs{q^2}$. Since $o(H)|o(G)$, it is
immediate that $H<G$. In this case, $\sigma$ is now a group homomorphism
on the cyclic group $G$ with a kernel of order $r$. It is immediate that
$\sigma$ has the form claimed.

Now suppose $G$ is not cyclic, so that $G=S_l\cup S_l h$ where $o(G)=2l$.
Note that $K$ is the kernel
of the restriction of $\sigma$ to $S_l$ and $S\approx S_l/K$. This also means
$S$ is a subgroup of $S_l$, as both are cyclic subgroups of the groups of
non-zero squares of $\ff{q^2}$.

Suppose first that $S=H$. Then $H\approx S_l/K$, from which we find
$o(H)|l$ and so $H<S_l<G$. Since $2l=r\, o(H)$, we also have $2|r$ and
$o(K)=r/2$.
Again, we appeal to the known structure of homomorphisms on cyclic groups
to get that the restriction of $\sigma$ to $S_l$ satisfies
$\sigma(x)=x^{r j/2}$ for some integer $j$ relatively prime to $o(H)$.

Now suppose $S\ne H$.
If $\sigma(h)\in S$, then $\sigma(G)=S\cup S\sigma(h) = S\cup S =S\ne H$,
contradicting $\sigma$ being onto.
Thus, $\sigma(h)\notin S$, and $S\sigma(h)$ is a nontrivial coset of $S$.
It follows that $H=S\cup S\sigma(h)$, from which $[H:S]=2$. This time we have
$o(K)=r$, and the restriction of $\sigma$ to $S_l$ satisfies
$\sigma(x)=x^{r j}$ for some integer $1\le j< o(H)/2$ with
$\gcd(j,o(H)/2)=1$.
\end{proof}

\section{Nearfield planes and their collineations}

There are a number of ways in which to construct a projective plane from a
nearfield. We follow the construction and notation of Wantz \cite{W-2008-uitrn},
whose results motivated this work.
\begin{lem}
Let $\nearf{F}=\cyc{\nearf{F},+,\nm}$ be any nearfield.
We define a point set $P$ and line set $L$ by
\begin{align*}
P &= \{(x,y,1) \,:\, x, y\in\nearf{F}\}
\cup \{(1,y,0) \,:\, y \in \nearf{F}\} \cup \{ (0,1,0) \},\\
L &= \{[s,1,t] \,:\, s,t \in\nearf{F}\}
\cup \{ [1,0,t] \, : \, t \in \nearf{F} \} \cup \{[0,0,1] \},
\end{align*}
respectively. If incidence $\Inc$ is defined by
\begin{equation*}
(x,y,z)\, \Inc\, [s,u,t] \text{ if and only if }
x \nm s + y \nm u + z \nm t = 0,
\end{equation*}
then $\NP(\nearf{F})=(P,L,\Inc)$ is a projective plane of order $o(\nearf F)$.
\end{lem}
This construction is standard (and not limited to nearfields). Note that in
this construction the {\em line at infinity} is considered to be the line
$$[0,0,1]=\{(1,y,0) \,:\, y \in \nearf{F}\} \cup \{ (0,1,0) \},$$
with all points not lying on $[0,0,1]$ viewed as {\em affine points}.
Moreover, the line at infinity is a translation line, see Dembowski
\cite{bD-1968-fg}, pages 127--131, and Section 5.1 for further
information on this.

With regard to the collineation groups of nearfield planes, we have the
following results from Andr\'e \cite{A-1955-peuf}.
\begin{lem}[{\cite{A-1955-peuf}}, Satz 12] \label{andresatz12}
Let $\nearf{F}=\cyc{\nearf{F},+,\nm}$ be any regular nearfield of order
$q=p^e\ne 9$, and let $\NP(\nearf F)$ be the corresponding nearfield plane.
The full collineation group $\Gamma$ of $\NP(\nearf F))$ is given by
\begin{equation*}
\Gamma = \grp{T}\, \hat{\grp{R}}_0 \, \widetilde{\grp{U}}_0\, \grp{B}_0,
\end{equation*}
where
\begin{align*}
\grp{T} &= \{ (x,y,1)\mapsto (x+u,y+v,1) \, : \, u,v\in \nearf F\},\\
\hat{\grp{R}}_0 &= \{ (x,y,1)\mapsto (x \nm c,y\nm d,1)
\, : \, c,d \in \nearf{F}^\star\},\\
\widetilde{\grp{U}}_0 &= \{ (x,y,1)\mapsto (\tau(x) ,\tau(y),1) \, : \,
\tau \in \Aut(\nearf F)\}, (replace\ \sigma \ by \ \tau) \text{ and}\\
\grp{B}_0 &= \{ \id, (x,y,1)\mapsto (y,x,1)\}.
\end{align*}
Furthermore, we have $o(\grp T)=q^2$,
$o(\hat{\grp{R}}_0)= (q-1)^2$, and
$o(\widetilde{\grp{U}}_0)=e$, so that $o(\Gamma)=2 e q^2(q-1)^2$.
\end{lem}
From this result, Andr\'e immediately obtains the following very useful
description of $\Gamma$.
\begin{lem}[{\cite{A-1955-peuf}}, Satz 13] \label{andresatz13}
Any collineation of the nearfield plane $\NP(\nearf F)$ is of the form
$$(x,y,1)\mapsto (\tau(x) \nm c+u,\tau(y) \nm d+v,1),$$
or
$$(x,y,1)\mapsto (\tau(y) \nm  d+v,\tau(x) \nm  c+u,1),$$
where $\tau \in \Aut(\nearf F)$, and $c,d,u,v \in \nearf{F}$, $c,d\ne 0$.
\end{lem}
A {\em linear collineation} of $\NP(\nearf F)$ is any collineation not
induced by an automorphism of $\nearf F$. \Cref{andresatz12} thus tells us
the group of all linear collineations is precisely
$\grp{T}\, \hat{\grp{R}}_0 \, \grp{B}_0$. Following the notation of
\cite{W-2008-uitrn}, we define the following mappings to describe all linear
collineations: for all $c,d,u,v,s,t \in\nearf F$ with $c,d\ne 0$,
\begin{itemize}
\item Define $\phi(c,d,u,v)$ on $P$ and $L$ by
\begin{align*}
(x, y, 1) &\mapsto (x \nm c + u, y \nm d + v, 1)\\
(1, y, 0) &\mapsto (1, c^{-1} \nm y \nm d, 0)\\
(0, 1, 0) &\mapsto (0, 1, 0) \\
[s, 1, t] &\mapsto [c^{-1} \nm s \nm d, 1, t \nm d - u \nm c^{-1} \nm s \nm d - v] \\
[1,0, t] &\mapsto [1, 0, t \nm c - u]\\
[0, 0, 1] &\mapsto [0, 0, 1]
\end{align*}
\item Define $\gamma(c,d,u,v)$ on $P$ and $L$ by
\begin{align*}
(x, y, 1) &\mapsto (y \nm c + u, x \nm d + v, 1)\\
(1, y, 0) &\mapsto (1, (y \nm d)^{-1} \nm c, 0), y\ne 0 \\
(1, 0, 0) &\mapsto (0, 1, 0) \\
(0, 1, 0) &\mapsto (1, 0, 0) \\
[s, 1, t] &\mapsto [c^{-1} \nm s^{-1} \nm d, 1, t \nm s^{-1} \nm d - u \nm c^{-1} \nm s^{-1} \nm d - v], s\ne 0 \\
[0,1,t] &\mapsto [1,0, t \nm c - u]\\
[1,0, t] &\mapsto [0, 1, t \nm d - v]\\
[0, 0, 1] &\mapsto [0, 0, 1]
\end{align*}
\end{itemize}
Clearly, we have
\begin{align*}
\grp{T} &= \{\phi(1,1,u,v) \,:\, u,v\in\nearf F\}, \text{ and}\\
\hat{\grp{R}}_0 &= \{\phi(c,d,0,0)\,:\, c,d\in\nearf F^\star\}.
\end{align*}
We note that $\grp{T}$ is the translation group of the plane; it is easily
observed that it consists of all $q^2$ elations of the plane with
axis $[0,0,1]$. Thus $\grp{T}\approx\cyc{\ff{q},+}\times\cyc{\ff{q},+}$.
The following result will prove useful.
\begin{lem}\label{philemma}
Let $\nearf F$ be a nearfield of order $q=p^e$ for some prime $p$.
The following statements hold.
\begin{enumerate}[label=(\roman*)]
\item The group of all elations of $\NP(\nearf F)$ with center $(0,1,0)$ and
axis $[0,0,1]$ is given by
$$H_0 = \{\phi(1,1,0,v)\,:\, v\in\nearf F\}.$$
Moreover, any subgroup of $H_0$ is of the form
$\{\phi(1,1,0,v)\,:\, v\in V\}$ for some subspace $V$ of
$\ff{q}$, when viewed as a vector space over $\ff{p}$.
\item The group of all elations of $\NP(\nearf F)$ with center $(1,y,0)$,
$y\ne 0$, and axis $[0,0,1]$ is given by
$$H_y = \{\phi(1,1,u,u\nm y)\,:\, u\in\nearf F\}.$$
\item Any linear collineation of $\NP(\nearf F)$ of order a power of $p$ must
lie in $\grp{T}$.
\end{enumerate}
\end{lem}
\begin{proof}
It is easily verified that $H_y$ has the claimed center for all $y$. Since
there are at most $q$ elations with any given center and axis, it is clear
we have all such elations in each case.  All that
needs to be established to complete the proof of (i) or (ii) is the form of
$V$ for a subgroup of $H_0$, but this follows at
once from the observation that composition of elements of $H_0$ acts in the
same manner as field addition in the 4th coordinate of the $\phi$ involved.

Finally, it is easy to prove that
$\grp{T}$ is normal in the group of linear collineations
$\grp{T}\, \hat{\grp{R}}_0 \, \grp{B}_0$. Since $\grp{T}$ is a Sylow-$p$
subgroup and normal, it is unique. Thus, all elements of order a power of $p$
must lie in $\grp{T}$, proving (iii).
\end{proof}
For further information on nearfields and their planes, see Dembowski
\cite{bD-1968-fg}, pages 229--232.

\section{Central collineations that fix a unital}

Before working with unitals in nearfield planes, we first prove a general
result concerning central collineations that fix a unital in any finite
projective plane.
\begin{thm}  \label{centralcollineationlemma}
Let $U$ be a unital of order $m$ embedded in a projective plane
$\proj{P}$ of order $m^2$.
Suppose $G=\Aut(U)$ contains a non-trivial central collineation $\phi$ with
axis $\lne{l}$.
The following statements hold.
\begin{enumerate}[label=(\roman*)]
\item $\phi$ is an elation if and only if $\lne{l}$ is a tangent line of $U$.
\item $\phi$ is an homology if and only if $\lne{l}$ is a secant line of $U$.
\end{enumerate}
\end{thm}
\begin{proof}
Choose a point $\pnt{P}\in\lne{l}\cap U$ and let $\lne{b}$ be the unique
tangent line through $\pnt{P}$.
As $\phi\in G$, we must have $\lne{b}^\phi=\lne{b}$.
Recall that, under the action of $\phi$, the non-fixed lines through $\pnt{P}$
must be partitioned into orbits of size $o(\phi)$.

There are two cases.
\begin{itemize}[label=$\boxdot$]
\item $\lne{l}$ is a tangent line.\\
In this case, we have $\lne{l}=\lne{b}$.
Now, if $\pnt{P}$ is the center of $\phi$, then $\phi$ is an elation.
Otherwise, there are $m^2$ lines through $\pnt{P}$ that are not fixed by
$\phi$. Thus, $o(\phi)|m^2$, which again forces $\phi$ to be an elation.

\item $\lne{l}$ is a secant line.\\
In this case, we have $\lne{l}\ne\lne{b}$.
Since $|\lne{l}\cap U|=m+1\ge 3$, we may choose a point
$\pnt{Q}\in\lne{l}\cap U$ that is neither $\pnt{P}$ nor the center of $\phi$.
Let $\lne{b}'$ be the unique tangent line through $\pnt{Q}$.
We still have $\lne{l}\ne\lne{b}'$, so that there are
$m^2-1$ lines through $\pnt{Q}$ that are not fixed by
$\phi$. Thus, $o(\phi)|(m^2-1)$, which forces $\phi$ to be an homology.
\end{itemize}
\end{proof}
This simple result appears to be new. For example, it is not found in the
primary book on unitals by Barwick and Ebert \cite{bBE-2008-uipp}, nor is it
found in the seminal texts of Dembowski \cite{bD-1968-fg} or Hughes and Piper
\cite{bHP-1973-pp}.

Recall that the group of translations of a plane $\proj{P}$ is the group of
elations fixing a specified line at infinity.
We get the immediate corollary.
\begin{cor} \label{cor1}
Let $\proj{P}$ be a projective plane of order $m^2$ with a non-trivial translation group
$\grp{T}$.
If $\proj{P}$ admits a unital $U$ of order $m$ with
$\Aut(U)\cap\grp{T}$ non-trivial, then $U$ is a parabolic unital.
\end{cor}
\begin{proof}
Under the hypothesis, $\Aut(U)$ contains an elation with axis $[\infty]$.
\Cref{centralcollineationlemma} now implies $[\infty]$ is a tangent line, so
that $U$ is parabolic.
\end{proof}
For nearfields, we can say a little more.
We use $\ordp{n}$ to denote the $p$-order of
$n$; that is, the largest power of $p$ that divides $n$.
\begin{lem} \label{onlyparabolic}
Let $\nearf F$ be a nearfield of order $q^2=p^{2e}$.
Let $U$ be a unital of order $q$ embedded in $\NP(\nearf F)$ and set
$G=\Aut(U)$.
If $\ordp{o(G)}>\ordp{e}$, then $U$ is a parabolic unital and $G\cap\grp{T}$ is
a non-trivial subgroup of $H_y$ of order $p^n$ where $n=\ordp{o(G)}-\ordp{e}$.
\end{lem}
\begin{proof}
Set $H=G\cap\grp{T}$ and consider $S\in\Syl_p(G)$. Note that
$o(S)=p^o$ where $o=\ordp{o(G)}$.
By \Cref{andresatz12}, the elements of $S$ can only
come from $H$ or $G\cap\hat{\grp{U}}_0$ when $p|e$. In the latter case, there are
at most $p^{\ordp{e}}$ such elements of order a power of $p$.
Since $\ordp{o(G)}>\ordp{e}$ by hypothesis, $H$ must be non-trivial; indeed,
it's order must be $p^n$.
Thus, $G$ contains non-trivial elations with axis $[0,0,1]$.
\Cref{cor1} now tells us that $U$ is parabolic.

Let $\pnt{P}$ be the unique point on $[0,0,1]$ contained in $U$.
If $\phi\in H$ has center $\pnt{C}\ne \pnt{P}$,
then it fixes each affine tangent of $\pnt{C}$ and hence each tangent point of these tangents.
But this is
impossible as $\phi$ can fix no points off its axis. Thus, the center of $\phi$
is $\pnt{P}$. \Cref{philemma} now forces $H<H_y$ for some $y\in\N$.
\end{proof}

\section{Restrictions on the linear collineations of unitals in $\NP$}

For convenience, we use $\N=\N(2,q)$ and $\NP=\NP(\N)$ throughout the
remainder of this article. Note that $\N$ and $\NP$ have order $q^2$, with
$q=p^e$ for some odd prime $p$.
Let $\Tr$ and $\Nm$ be the trace and norm functions, respectively, of
$\ff{q^2}$ onto $\ff{q}$.
That is, $\Tr(x)=x^q+x$ and $\Nm(x)=x^{q+1}$.
We recall that, for all $x,y\in\ff{q^2}$, we have
$\Tr(x^q)=\Tr(x)$, $\Tr(x+y)=\Tr(x)+\Tr(y)$, $\Nm(x^q)=\Nm(x)$, and
$\Nm(xy)=\Nm(x)\Nm(y)$.
By the properties of the norm, we also have $\Nm(x \nm y)=\Nm(x)\Nm(y)$ for all
$x,y\in\ff{q^2}$, so that the norm function acts multiplicatively with
respect to both field and nearfield multiplication.
It is well understood that the trace is $q$-to-1 from $\ff{q^2}$ onto
$\ff{q}$, while the norm is $(q+1)$-to-1 from $\ffs{q^2}$ onto $\ffs{q}$.

We now fix a primitive element $\gen$ of $\ff{q^2}$, so that
$\ffs{q^2}=\cyc{\gen}$.
Set $\beta=\Nm(\gen)=\gen^{q+1}$ and $\epsilon=\gen^{(q+1)/2}$.
By construction, $\beta\in\nonsquare_q$, $\epsilon^q=-\epsilon$, and
$\epsilon^2=\beta$.
We also have $\ff{q^2}=\ff{q}(\epsilon)$, so that any $a\in\ff{q^2}$ can
be written as $a=a_1+a_2\epsilon$ with $a_i\in\ff{q}$.
Furthermore, $\Tr(a)=2a_1$ and $\Nm(a)=a_1^2-a_2^2\beta$.

We begin with the following consequences of the previous sections.
\begin{lem} \label{reductionlem}
Let $U$ be a unital of order $q$ embedded in $\NP$.
Set $G=\Aut(U)\cap\grp{T}\, \hat{\grp{R}}_0 \, \grp{B}_0$ and
$H=G\cap\grp{T}$, so that $G$ is the group of linear collineations fixing $U$
and $H$ is the group of translations fixing $U$.
Fix $n=\ordp{o(G)}$ and suppose $n>0$.
The following statements hold.
\begin{enumerate}[label=(\roman*)]
\item $U$ is a parabolic unital.
\item $H<H_y$ for some $y\in\N$ with $o(H)=p^n$.
\item If $U$ contains $(0,1,0)$ or $(1,0,0)$,
then $\NP$ admits a unital $U'$ isomorphic to $U$
containing $(0,1,0)$, and
$\Aut(U')\cap\grp{T}=\{\phi(1,1,0,w)\,:\, w \in W\}$ for some subspace
$W$ of $\ff{q^2}$ with $|W|=p^n$.
\item If $U$ does not contain $(0,1,0)$ or $(1,0,0)$,
then $\NP$ admits a unital $U'$ isomorphic to $U$
containing $(1,1,0)$, so
that $\Aut(U')\cap\grp{T}=\{\phi(1,1,w,w)\,:\, w \in W\}$ for some subspace
$W$ of $\ff{q^2}$ with $|W|=p^n$.
\item $n\le e$.
\end{enumerate}
\end{lem}
\begin{proof}
By construction, $H\in\Syl_p(G)$, so that $o(H)=p^n$.
Thus, $G$ contains non-trivial central collineations with
axis $[0,0,1]$.
\Cref{cor1} now tells us that $U$ is parabolic.
It therefore contains either $(0,1,0)$ or $(1,y,0)$ for some $y\in\N$.
In either case, \Cref{onlyparabolic} proves (ii).

The orbit structure of the action of $\Aut(\NP)$ on $[0,0,1]$ is
$\{(0,1,0),(1,0,0)\}$ and $[0,0,1]\setminus \{(0,1,0),(1,0,0)\}$.
Thus, depending on which orbit the unital $U$ intersects, we can always
find an isomorphic unital $U'$ satisfying the conditions of (iii) or (iv).
The remainder of (iii) and most of (iv) now follows at once from
\Cref{philemma}.
The only part that of (iv) that needs
verification is that $W$ is a subspace of $\ff{q^2}$, but this follows in a similar manner
to proving the form of $H_0$ in \Cref{philemma}.

To prove (v), 
choose any line $\lne{l}\ne[0,0,1]$ through the unique unital point $\pnt{P}$ on $[0,0,1]$.
Since $[0,0,1]$ is
the unique tangent line through $\pnt{P}$, $\lne{l}$ must be a secant line.
Consider the action of $H$ on the $q+1$ unital points of $\lne{l}$. 
This action
fixes the point $\pnt{P}$ and must partition the remaining $q$ points into
orbits of size $o(H)$, which proves $o(H)\le q$.
By (ii), $o(H)=p^n$. Since $q=p^e$, we have $n \leq e$.
\end{proof}
We now want to focus on the maximal case described in \Cref{reductionlem}(v), and will now obtain substantial restrictions
on the group of linear collineations fixing the unital in that case. We start
with the unitals of \Cref{reductionlem} (iii).

\begin{thm} \label{restrictionthm}
Let $U'$ be a unital of order $q$ of $\NP$ that contains $(0,1,0)$ or $(1,0,0)$.
If $q|o(\Aut(U')\cap\grp{T}\hat{\grp{R}}_0\grp{B}_0)$,
then $U'$ is isomorphic to a unital $U$ for which the following statements hold.
\begin{enumerate}[label=(\roman*)]
\item $U$ is a parabolic unital of $\NP$ containing $(0,1,0)$.

\item The group $H=\Aut(U)\cap\grp{T}$ satisfies
\begin{equation*}
H=\{\phi(1,1,0,w)\,:\, w \in W\},
\end{equation*}
where $W$ is a subspace of $\ff{q^2}$, when viewed as a vector
space over $\ff{p}$, with $o(W)=q$.
Additionally, $W$ must have a scalar field $\ff{q_0}$ for some
subfield of $\ff{q}$.

\item $U\cap[1,0,0] = \{(0,w,1) \,:\, w\in W\}$. Furthermore, this is precisely
the set of tangency points of the $q$ affine tangents of $(1,0,0)$.

\item The group $G=\Aut(U)\cap\grp{T}\hat{\grp{R}}_0\grp{B}_0$ satisfies
$G=HG_1$, where $G_1=G\cap\hat{\grp R}_0$. Furthermore,
$q|o(G)$ and $o(G)|q(q^2-1)$,

\item $H\norm G$.

\item $o(G_1)=o(G)/q$.

\item The set $C=\{c\in\N\,:\, \phi(c,d,0,0)\in G_1\}$ has
cardinality $o(G_1)$ and is a subgroup of $\cyc{\N^\star,\star}$.

\item The set $D=\{d\in\N\,:\, \phi(c,d,0,0)\in G_1\}$
is a subgroup of $\ffs{q_0}$. In particular, $D < \ffs{q}$.

\item The mapping $\delta:C\rightarrow D$ defined by $\delta(c)=d$ if
$\phi(c,d,0,0)\in G_1$ is an onto $r$-to-1 homomorphism, where
$r=o(C)/o(D)$. Furthermore, $r|(q+1)$.

\end{enumerate}
\end{thm}
\begin{proof}
By hypothesis, $U'$ satisfies the conditions of \Cref{reductionlem}. Thus,
$U'$ is necessarily isomorphic to a parabolic unital $U''$ containing $(0,1,0)$.
Let $(a,b,1)$ be one of the tangency points of a tangent line on the non-unital
point $(1,0,0)$ for $U''$.
Set $U=U''^{\phi(1,1,-a,-b)}$.
Then $U$ is a parabolic unital of $\NP$ containing $(0,1,0)$ isomorphic to
$U'$, which yields (i).

The unital $U$ is in fact a version of the unital described in
\Cref{reductionlem} (iii), which gives $o(H)=q$.
The only part of (ii) left to prove is that the
scalar field of $W$ is a subfield of $\ff{q}$. To this end,
set $\ff{q_0}$ to be the subfield of $\ff{q^2}$ that acts as the scalars for
$W$, and let $W$ have dimension $j$ over $\ff{q_0}$.
Then $q=q_0^j$, so that $\ff{q_0}$ is necessarily a subfield of $\ff{q}$.

For (iii), as $(a,b,1)^{\phi(1,1,-a,-b)}=(0,0,1)\in U''$, the orbit of
$\{(0,w,1)\,:\, w\in W\}$ under $H$ is the set of $q$ tangency points of the
tangent lines of $U$ through $(1,0,0)$.
Since this orbit lies on $[1,0,0]$, the group $G$ fixes $[1,0,0]$.
Consequently,
\begin{equation*}
G\cap \grp{T} \hat{\grp R}_0 <
\{ \phi(c,d,0,v) \,: \,c,d\in\N^\star\land v\in\N \}.
\end{equation*}
Since $o(W)=q$, this orbit is exactly the set of $U$-unital points on $[1,0,0]$,
as claimed.

Now set $S=\{s \in\N \,: \,\phi(c,d,0,s) \in G\}$. Clearly, $W\subseteq S$.
On the other hand, $(0,0,1)^{\phi(c,d,0,s) }=(0,s,1)$ is an affine unital point
on $[1,0,0]$ for any $\phi(c,d,0,s) \in G$. So, $S\subseteq W$ and $S=W$.
It follows that $G=H G_1$.
By \Cref{andresatz12}, we must have $G_1 < \hat{\grp{R}}_0$, and so
$o(G_1)|(q^2-1)$.
Since $H\cap G_1=\{id\}$, we have $q|o(G)$ and $o(G)|q(q^2-1)$, establishing
(iv).

For (v), let $\varphi=\phi(1,1,0,w)$ for some $w\in W$ and
choose any $\rho=\phi(c,d,0,v)\in G$. Now, $\rho^{-1}=\phi(c^{-1},d^{-1},0,v')$ with
$v'=-v\nm d^{-1}$. One checks that $(x,y,1)^{\rho^{-1} \varphi \rho}=(x,y+w\nm d,1)$,
from which we find $\rho^{-1} \varphi \rho = \phi(1,1,0,w\nm d)\in\mathfrak T$.
However, since $\rho^{-1} \varphi \rho\in G$, we have $\rho^{-1} \varphi \rho\in H$, and $H\norm G$.
We also have $W=\{w\nm d\,:\, w\in W\}$ anytime $\phi(c,d,0,v)\in G$.

To prove (vi), we already know that
$G_1<\hat{\grp{R}}_0$, so that
$G_1$ consists only of collineations of the form $\phi(c,d,0,0)$.
The order of $G_1$ is clear since $o(H)=q$.

For (vii), since $C$ is finite, to prove it is a subgroup of
$\cyc{\N^\star,\star}$ it suffices to prove $C$ is closed,
but this follows immediately from the fact $G_1$ is a group.
This also proves $D$ is a subgroup of $\cyc{\N^\star,\star}$.
To determine $o(C)$, let $U_A$ denote the $q^3-q$ points in $U$ that do not
lie on $[1,0,0]$. Every such point is of the form $(x,y,1)$ and it is easily
observed that $G$ can fix no such point. Thus, the action of $G$ on $U_A$
determines $r$ orbits of size $o(G)$.
Now let us consider the lines $[1,0,z]$ with $z\ne 0$.
Each of these lines is a secant line of $U$, since $(0,1,0)\in [1,0,z]$ and
$[0,0,1]$ is the unique tangent line through $(0,1,0)$.
Thus, each line $[1,0,z]$ contains $q$ points of $U_A$.

Let $O$ be any orbit of the action of $G$ on $U_A$, and let us
suppose $(-z,y,1)\in O\cap [1,0,z]$. Then
$(-z,y,1)^{\phi(1,1,0,v)}=(-z,y+v,1)\in O\cap [1,0,z]$, so that
$ O\cap [1,0,z] = U_A\cap [1,0,z]$; i.e. $ O$ contains all of the
points in $U_A$ lying on $[1,0,z]$.
Now, for each $c\in C$, we have $[1,0,z]^{\phi(c,d,0,v)}=[1,0,z\star c]$, so
that the number of points in $ O$ is $q\, o(C)$. Thus, $o(G)=q\, o(C)$.
But $o(G)=q\, o(G_1)$ by (vi), and so $o(G_1)=o(C)$.

For the claimed restrictions on $D$ in (viii), we know
from our proof of (v) that any $d$ for which $\phi(c,d,0,v)\in G$
satisfies $\{w\nm d\,:\, w\in W\}=W$.
If $D<\square_{q^2}^\star$, then $W\nm d=dW=W$ for all $d\in D$, so that $D$ is
contained in the scalar field of $W$. Thus, $D<\ffs{q_0}$.
We now show that $D\cap\nonsquare_{q^2}=\varnothing$.
Suppose otherwise.
Then, by \Cref{subgroups}, we necessarily have $D=S\cup S h$ for some
$S<\square_{q^2}^\star$ and $h\in\nonsquare_{q^2}$ satisfying
$\Nm(h)\in S$.
Note that $\Nm(h)\in\nonsquare_q$, as $-1=h^{(q^2-1)/2}=\Nm(h)^{(q-1)/2}$.
Thus, $S\cap\nonsquare_q\neq\varnothing$.
We first note that for any $w\in W^\star$, $|w\nm D|=o(D)$. For this to be
false, $w\star h=w^q h=w s$ for some $s\in S$. Thus, $h=s w^{1-q}$, which is
a square of $\ffs{q^2}$, a contradiction.
Now fix $w_1,w_2\in W$. If $w_1\nm D$ and $w_2\nm D$ are not disjoint, then
there exists $d_1,d_2\in D$ for which $w_1\nm d_1=w_2\nm d_2$.
Thus, $w_1\nm d_1\nm D=w_2\nm d_2\nm D$, and simplifying yields
$w_1\nm D=w_2\nm D$. This proves that $w_1\nm D$ and $w_2\nm D$ are either
equal or disjoint.
Since $|w\nm D|=o(D)$ and $W\nm d=W$, we find $o(D)$ divides $|W^\star|=q-1$.
Thus, $o(S)|(q-1)/2$, which forces $S<\square_q^\star$ by the uniqueness
of subgroups of cyclic groups. But this contradicts
$S\cap\nonsquare_q\neq\varnothing$. We conclude that
$D\cap\nonsquare_{q^2}=\varnothing$, and so $D<\ffs{q_0}$ proving (viii).

Finally, that $\delta$ is an onto homomorphism follows at once from the
definition of $\delta$ and the fact $G_1$ is a group.
Equally obvious is that the kernel of $\delta$ is
$$\Ker(\delta)=\{c\in C\,:\, \phi(c,1,0,0)\in G_1\}.$$
Setting $r=o(\Ker(\delta))=o(C)/o(D)$, it is now immediate that $\delta$ is
$r$-to-1. It remains to prove $r|(q+1)$. To this end,
let $[0,1,z]$ be any secant line and consider the orbits of the action of
the group $\{\phi(c,1,0,0)\,:\, c\in\Ker(\delta)\}$
on the $q+1$ unital points on $[0,1,z]$.
Each such orbit has size $r$, and so $r|(q+1)$.
\end{proof}
It turns out that the maximal case stemming from \Cref{reductionlem} (iv) and
(v) can be dealt with in a similar fashion as in \Cref{restrictionthm}, but the
outcome is far more restrictive.
\begin{thm} \label{restrictionthm1y0}
Let $U'$ be a unital of order $q$ of $\NP$ that does not contain $(0,1,0)$ or $(1,0,0)$.
If $q|o(\Aut(U')\cap\grp{T}\hat{\grp{R}}_0\grp{B}_0)$,
then $U'$ is isomorphic to a unital $U$ for which the following statements hold.
\begin{enumerate}[label=(\roman*)]
\item $U$ is a parabolic unital of $\NP$ containing $(1,1,0)$.

\item The group $H=\Aut(U)\cap\grp{T}$ satisfies
\begin{equation*}
H=\{\phi(1,1,w,w)\,:\, w \in W\},
\end{equation*}
where $W$ is a subspace of $\ff{q^2}$, when viewed as a vector
space over $\ff{p}$, with $o(W)=q$.
Additionally, $W$ must have a scalar field $\ff{q_0}$ for some
subfield of $\ff{q}$.

\item $U\cap[-1,1,0] = \{(w,w,1) \,:\, w\in W\}$. Furthermore, this is precisely
the set of tangency points of the $q$ affine tangents of $(1,0,0)$.

\item The group $G=\Aut(U)\cap\grp{T}\hat{\grp{R}}_0\grp{B}_0$ satisfies
$G=H$.
\end{enumerate}
\end{thm}
\begin{proof}
The proof is, for the most part, very similar to that of \Cref{restrictionthm}.
By hypothesis, $U'$ satisfies the conditions of \Cref{reductionlem} (iv), and
so $U'$ is isomorphic to a parabolic unital $U''$ containing $(1,1,0)$.
Set $G=\Aut(U)\cap\grp{T}\hat{\grp{R}}_0\grp{B}_0$ and
$G_1=G\cap\hat{\grp R}_0$.
Let $(a,b,1)$ and $U=U''^{\phi(1,1,-a,-b)}$ be as in the proof of
\Cref{restrictionthm}.  The orbit of
$\{(w,w,1)\,:\, w\in W\}$ of $(0,0,1)\in U'$ under $H$ is the set of $q$ tangency points of the
tangent lines of $U$ through $(1,0,0)$.
Since this orbit lies on $[-1,1,0]$, the group $G$ fixes $[-1,1,0]$.
By considering \Cref{andresatz12}, we have
\begin{equation*}
G\cap \grp{T} \hat{\grp R}_0 <
\{ \phi(c,c,v,v) \,: \,c\in\N^\star\land v\in\N \}.
\end{equation*}
Indeed, if $[-1,1,0]=[-1,1,0]^{\phi(c,d,u,v)}=[c^{-1} \nm (-d), 1, - u \nm c^{-1}  \nm (-d) - v]$, then $c^{-1} \nm (-d)=-1$ and $- u \nm c^{-1}  \nm (-d) - v=0$, giving $d=c$ and $u=v$.

Now set $S=\{s \in\N \,: \,\phi(c,c,s,s) \in G\}$. Deducing that
$S=W$ and $G=H G_1$ now follows similarly to the previous proof.
We also have $G_1=\{\phi(c,c,0,0)\,:\, c\in C\}$ for some group $C$
isomorphic to $G_1$.
To prove (iv), we shall now show $G_1$ is trivial.

Consider the point $(0,1,0)$. Since this point does not lie in $U$, we know
there are $q+1$ tangent lines through $(0,1,0)$. One of these lines is
necessarily $[0,0,1]$, which is fixed by $G$. The remaining $q$ tangent lines
each contain a unique point of $U$. Let $T$ be the set of these tangency points.
Since $(0,1,0)$ is fixed by $G$, the action of $G$ on $U$ must
map $T$ to itself.
Fix a point $(x,y,1)\in T$. Note that $x\neq y$ by (iii). The line through $(x,y,1)$ and $(1,1,0)$ is
$[-1,1,x-y]$. For any $\rho=\phi(1,1,w,w)$ with $w\in W$, we have
$(x,y,1)^\rho=(x+w,y+w,1)\in [-1,1,x-y]$. Since $o(W)=q$, all $q$ points of
$T$ lie on $[-1,1,x-y]$, and $T$ is an orbit of $G$.
In particular, $G$ must fix the line $[-1,1,x-y]$.
Suppose $\varphi=\phi(c,c,0,0)\in G_1$. Then
$[-1,1,x-y] = [-1,1,x-y]^\varphi = [-1,1,(x-y)\star c]$, which forces $c=1$.
Hence, $G_1$ is trivial and $G=HG_1=H$.
\end{proof}

\section{The non-existence of certain O'Nan configurations in unitals in $\NP$}

\begin{defn} \label{onanconfig}
An {\em O'Nan configuration} in a unital is a collection of four distinct lines
meeting in six distinct unital points.
\end{defn}
In 1972, O'Nan \cite{O-1972-aoubd} noted that these configurations do not occur
in the classical unital. It was later conjectured by Piper \cite{P-1979-ubd}
that any unital without such a configuration is the classical unital.
Recently, Stroppel \cite{S-2024-uwoca} showed that
the conjecture is true if for each unital point, there is an automorphism of
the unital that fixes the point and every block through it.

\begin{lem}\label{lem-collinear}
In $\NP$, three points $(x,y,1), (x_1,y_1,1), (x_2,y_2,1)$ are collinear on a
line $[u,1,z]$ with $u\ne 0$, if and only if
$(x-x_2)\nm (x-x_1)^{-1}=(y-y_2)\nm(y-y_1)^{-1}.$
If these products are in $\ff{q}$, then
$\frac{x - x_2}{x - x_1} = \frac{y - y_2}{y - y_1}.$
\end{lem}
\begin{proof}
The line determined by $(x,y,1)$ and $(x_1,y_1,1)$ is
$[-(x -x_1)^{-1} \nm (y - y_1), 1, x \nm (x -x_1)^{-1} \nm (y-y_1)-y]$.
For $(x_2,y_2,1)$ to be on this line, we must have
$$- x_2 \nm (x - x_1)^{-1} \nm (y - y_1) + y_2
+ x \nm (x - x_1)^{-1} \nm (y-y_1) - y = 0.$$
Simplifying gives $(x-x_2)\nm (x-x_1)^{-1}=(y-y_2)\nm(y-y_1)^{-1}$.
The final statement follows from the definition of $\nm$.
\end{proof}
We next prove a lemma concerning O'Nan configurations in $\NP$ containing
$(0,1,0)$.
\begin{lem}\label{lem-planeONan}
Consider an O'Nan configuration in $\NP$ whose six points are $(0,1,0)$,
$\pnt{P}=(x,y,1)$ and $\pnt{P_{ij}}=(x_i,y_{ij},1)$ with $y\neq y_{ij}$ for
$i,j=1,2$.
Suppose $\Tr(y_{11})=\Tr(y_{12})$.
Then $\frac{x - x_2}{x - x_1}\in \ff{q}$ if and only if $\Tr(y_{21})=\Tr(y_{22})$.
\end{lem}
\begin{proof}
Applying $\phi(1,1,-x,-y)$ if needed, without loss of generality, we may assume $\pnt{P}=(0,0,1)$.
By assumption, $x_i, y_{ij}\neq 0$ and $y_{i1}\neq y_{i2}$ for $i=1,2$.
Applying \Cref{lem-collinear} gives
\begin{equation}\label{eqn-onan}
x_2 \nm x_1^{-1}=y_{21}\nm y_{11}^{-1}=y_{22}\nm y_{12}^{-1}.
\end{equation}

Suppose $\Tr(y_{11})=\Tr(y_{12})$
and $\frac{x_2}{x_1}\in \ff{q}$.
Since $\Tr(y_{11})=\Tr(y_{12})$, we have $y_{12}=y_{11}+a$ for some $a\in \ff{q^2}$ with $\Tr(a)=0$.
Let $b=y_{22}-y_{21}$.
It follows from \eqref{eqn-onan} that
$\frac{y_{21}}{y_{11}}=\frac{y_{22}}{y_{12}}=\frac{y_{21}+b}{y_{11}+a}$.
On simplification,
$b=\frac{y_{21}}{y_{11}}a$
and
so $\Tr(b)=\Tr\left(\frac{y_{21}}{y_{11}}a\right)=\frac{y_{21}}{y_{11}}\Tr\left(a\right)=0$, where the second last equal sign holds because $\frac{y_{21}}{y_{11}} \in \ff{q}$.
So $\Tr(y_{22})=\Tr(y_{21}+b)=\Tr(y_{21})$.

Conversely, suppose $\Tr(y_{11})=\Tr(y_{12})$
and $\Tr(y_{21})=\Tr(y_{22})$.
Then for $i,j=1,2$,  $y_{ij}=s_i + t_{ij} \epsilon$ for some $s_i,t_{ij} \in \ff{q}$.
Recalling $\epsilon^q=-\epsilon$, we have
$y_{2j}\nm y_{1j}^{-1}
=\frac{s_2 \pm t_{2j}\epsilon}{s_1 \pm t_{1j}\epsilon}$, where the sign of $\pm$
depends on whether $y_{1j}$ is a square or not.
Consider the case when $y_{11}\in \square_{q^2}$ and $y_{12}\in\nonsquare_{q^2}$.
Equation \eqref{eqn-onan} gives
$\frac{s_2 + t_{21}\epsilon}{s_1 + t_{11}\epsilon}=\frac{s_2 - t_{22}\epsilon}{s_1 - t_{12}\epsilon}$, which is simplified to
$[s_1 (t_{21}+t_{22})-s_2(t_{11}+t_{12})]-(t_{12}t_{21}+t_{11}t_{22})\epsilon=0$.
So
$s_2=\frac{s_1 (t_{21}+t_{22})}{t_{11}+t_{12}}$ and
$t_{12}t_{21}=t_{11}t_{22}$.
Substituting these two equations one by one, we get
\[\begin{split}
y_{21}\nm y_{11}^{-1}
=\frac{s_2 + t_{21}\epsilon}{s_1 + t_{11}\epsilon}
=&\frac{\frac{s_1 (t_{21}+t_{22})}{t_{11}+t_{12}} + t_{21}\epsilon}{s_1 + t_{11}\epsilon}\\
=&\frac{ s_1 (t_{21}+t_{22}) + (t_{11} t_{21} +  t_{12} t_{21} )\epsilon }
{(t_{11}+t_{12})(s_1 + t_{11}\epsilon)}\\
=&\frac{ s_1 (t_{21}+t_{22}) + (t_{11} t_{21} + t_{11} t_{22})\epsilon }
{(t_{11}+t_{12})(s_1 + t_{11}\epsilon)}
=\frac{ (t_{21}+t_{22})(s_1 + t_{11}\epsilon) }
{(t_{11}+t_{12})(s_1 + t_{11}\epsilon)}
=\frac{ t_{21}+t_{22}}{t_{11}+t_{12}} \in \ff{q}.
\end{split}
\]
Similarly, $y_{21}\nm y_{11}^{-1}$ can be shown to be in $\ff{q}$ in the other 3 combinations of square or non-square of $y_{11},y_{12}$. In any case, $\frac{x_2}{x_1}$ is in $\ff{q}$ from \eqref{eqn-onan}.
\end{proof}

We can now prove our main result of this section.
\begin{thm}\label{thmONan}
Let $U$ be a parabolic unital of order $q$ in $\NP$ with $(0,1,0)\in U$ and where
$\{\phi(1,1,0,t \epsilon) \, :\, t \in \ff{q} \} < \Aut(U)$.
If $U$ contains an O'Nan configuration through $(0,1,0)$, then one of the lines
in the configuration must be $[0,1,z]$ for some $z\in\N$.
\end{thm}
\begin{proof}
Suppose $U$ contains an O'Nan configuration through $(0,1,0)$ and none of lines in the configuration are of the form $[0,1,z]$.
Let the coordinates of the points of the O'Nan configuration be as in
\Cref{lem-planeONan}.
Since $(x_i, y_{i1},1)$ is a unital point for $i=1,2$ and $\{\phi(1,1,0,t \epsilon) \, :\, t \in  \ff{q} \} < \Aut(U)$,
$(x_i, y_i,1)$ is a unital point if and only if $\Tr(y_i)=\Tr(y_{i1})$.
So $\Tr(y_{11})=\Tr(y_{12})$ and $\Tr(y_{21})=\Tr(y_{22})$.
By \Cref{lem-planeONan}, $\frac{x - x_2}{x - x_1}\in  \ff{q}$.

Let $S$ be the set of unital points on the lines in the subpencil from $\pnt{P}$ to the $q$ unital points $(x_1,y_1,1)$ where $\Tr(y_1)=\Tr(y_{11})$. So $S$ has $q^2+1$ points.
Let $A=\{x_3 \in \ff{q^2} : \frac{x - x_3}{x - x_1} \in  \ff{q}\}$.
By \Cref{lem-planeONan}, the second coordinate of the $q$ points in $[1,0,-x_3]\cap S$ have the same trace or $q$ different traces, depending whether $x_3$ is in $A$ or not.
Since $\{\phi(1,1,0,t \epsilon) \,:\, t \in  \ff{q} \} < \Aut(U)$,
if $x_3\notin A$, then $[1,0,-x_3]\cap S$ has exactly one unital point.
As $|A|=q$, this gives $q^2-q$ unital points in $S$.

Since $x_1$ and $x_2$ are in $A$, these $q^2-q$ points are distinct from the $2q$ points in $S$ which have the form
$(x_1, y_1,1)$ or $(x_2, y_2,1)$ where $\Tr(y_1)=\Tr(y_{11})$ and $\Tr(y_{2})=\Tr(y_{21})$.
Together with $\pnt{P}$, this gives at least $(q^2-q)+2q+1=q^2+q+1$ points in $S$, which is too many. Contradiction.
\end{proof}

\section{The sets $B(a,b)$ and the lines of $\NP$} \label{Babsets}

For $a,b\in\ff{q}$, we define
$$B(a,b)=\{(x,y,1) \in \NP \,:\, \Nm(x)=a \land \Tr(y)=b \}.$$
With the canonical form in \eqref{classicalU}, the classical unital is given by
$$\{ (0, 1, 0)\} \bigcup_{a\in \ff{q}} B(a,a).$$
This partition of unital points was studied by Taylor \cite{T-1974-ubd} and
Hui and Wong \cite{HW-2014-oeaub} for the extrinsic and intrinsic
characterizations, respectively, of a unitary polarity.
\begin{lem}\label{replacementlemma3}
The following statements hold in $\NP$.
\begin{enumerate}[label=(\roman*)]
\item Any line $[1,0,z]$ intersects $B(a,b)$ in $0$ or $q$
points.
\item For any $a\neq 0$, any line $[0,1,z]$ intersects $B(a,b)$ in $0$ or $q+1$
points.
\item Any line $[0,1,z]$ intersects $B(0,b)$ in $0$ or $1$
point.
\item Any line $[u,1,z]$, with $u\ne 0$, intersects $B(a,b)$ in $0, 1$ or $2$
points.
\end{enumerate}
\end{lem}
\begin{proof}
Firstly, all points on the line $[1,0,z]$ are of the form $(-z,y,1)$ with
$y\in\N$. Clearly $[1,0,z]$ does not intersect $B(a,b)$ if $a\ne\Nm(-z)$.
If $a=\Nm(-z)$, then there are $q$ choices of $y\in\N$ for which
$\Tr(y)=b$. A similar argument deals with the lines $[0,1,z]$.

Now consider a line $[u,1,z]$ with $u\ne 0$. The set of points of the form
$(x,y,1)$ lying on $[u,1,z]$ is given by
$$\{ (x,-z-x\star u,1) \,:\, x\in\N\}.$$
Suppose $(x,-z-x\star u,1)\in B(a,b)$. Then
$\Nm(x)=a$ and $\Tr(-z-x\star u)=b$.
We may rewrite the trace identity as $\Tr(u'x)=b+\Tr(z)$, where
$u'$ is $-u$ or $-u^q$ depending on whether $u\in\square_{q^2}$.
Set $x'=u'x$. Then we have the two equations
$\Tr(x')=b+\Tr(z)$ and $\Nm(u'^{-1}x')=\Nm(u')^{-1}\Nm(x')=a$.
Thus, to count the number of points in $[u,1,z]\cap B(a,b)$ it is enough to
count the number of solutions $x\in\ff{q^2}$ of $\Nm(x)=c_1$ and
$\Tr(x)=c_2$ for fixed $c_i\in\ff{q}$.
This is a well-known problem in finite fields. For want of a reference, we
now show that this system has 0, 1 or 2 solutions.
Using the representation $\ff{q^2}=\ff{q}(\epsilon)$, we
set $x=x_1+x_2\epsilon$, with $x_i\in\ff{q}$, and have the following system
of 2 equations in $x_1,x_2$:
\begin{align*}
x_1^2-x_2^2 \epsilon^2 &=c_1\\
2x_1 &=c_2.
\end{align*}
Clearly, we may solve for $x_1$ uniquely, leaving a quadratic over $\ff{q}$
in $x_2$, which has 2, 1, or 0 solutions if
$c_2^2-4c_1$ is in $\nonsquare_q$, 0, or in $\square_q$, respectively.
\end{proof}

\section{The all-ones polynomial over $\ff{q}$ -- an interlude}

For any $k\in\nat$, define $h_k\in\ffx{q}$ to be $h_k(X)=1+X+\cdots+X^k$.
For $x\ne 1$, it is clear $h_k(x)=(x^{k+1}-1)/(x-1)$. However, determining
necessary and sufficient conditions for when $h_k$ is a permutation polynomial
over $\ff{q}$ remained open for many years. It was only in 1994 that
Matthews resolved the problem.
\begin{lem}[{\cite{M-1994-ppotp}}, Theorem 1.1] \label{rexlem}
If $q=p^e$ is odd, $p$ a prime, then $h_k$ is a permutation polynomial over
$\ff{q}$ if and only if $k\equiv 1 \bmod (p(q-1))$.
\end{lem}
It turns out we shall also be interested in the behavior of $h_k$, but we will
need to know slightly more about the case where $h_k$ is not a permutation
polynomial.
If a non-constant polynomial $f\in\ffx{q}$ is known to not be a permutation
polynomial over $\ff{q}$ and $\deg(f)\le q-1$, then Wan \cite{W-1993-aalla}
showed
\begin{equation} \label{wanbound}
|f(\ff{q})| \le \left\lfloor q - \frac{q-1}{\deg(f)}\right\rfloor.
\end{equation}
This bound can be used to prove the following statement regarding the all-ones
polynomial.
\begin{lem} \label{allonespoly}
Set $E=\ff{q}\setminus\{0,1\}$.
If $1<k<q-1$ and $\gcd(k,q-1)=1$, then
there exists $c_1,c_2\in E$ such that
$h_k(c_1)=h_k(c_2)$.
\end{lem}
\begin{proof}
Clearly $k\not\equiv 1\bmod (p(q-1))$, so that $h_k$ is not a permutation
polynomial over $\ff{q}$ by \Cref{rexlem}. Since $k<q-1$, Wan's
bound, given in \Cref{wanbound}, shows
$$|h_k(\ff{q})| \le q-2.$$

Suppose $x\ne 1$ satisfies $h_k(x)=1$. As we have
\begin{equation*}
h_k(x) = \frac{x^{k+1}-1}{x-1},
\end{equation*}
we see $h_k(x)=1$ if and only if $x=0$ or $x^k=1$. Since $\gcd(k,q-1)=1$, there
are no solutions to $x^k=1$ with $x\ne 1$.
Thus, there are at most 2 pre-images of 1
under $h_k$, namely $x=0,1$. Thus, $1\in h_k(\ff{q})$ and $1\notin h_k(E)$,
which proves
$$|h_k(E)| \le q-3.$$
Since $|E|=q-2$, we must have two elements $c_1,c_2\in E$ with the same
image under $h_k$, as claimed.
\end{proof}

\section{Unitals containing $(0,1,0)$ in $\NP$ with the largest possible linear
collineation group}

We see from \Cref{restrictionthm} and \Cref{restrictionthm1y0} that the
largest possible linear collineation group size for a unital in $\NP$ is
$q(q^2-1)$, and this comes from the case outlined in \Cref{restrictionthm}.
In this section we shall concentrate on this case.
Our goal is to show that the only unitals in $\NP$ with an automorphism group
of order $q(q^2-1)$ are those determined by Wantz.

We start with some immediate consequences of the assumption that the
automorphism group of the unital is as large as possible.

\begin{lem} \label{CDWlem}
Let $U, G, H, W, C, D$ and $\delta$ be as in \Cref{restrictionthm}.
Assume $o(G)=q(q^2-1)$. Then
$C=\cyc{\N^\star,\star}=\square_{q^2}^\star\cup\square_{q^2}^\star \gen$ ,
$D=\ffs{q}$, and
$W=\ff{q}h$ for some $h\in\ffs{q^2}$.
\end{lem}
\begin{proof}
Since $o(G)=q(q^2-1)=q\, o(C)$ by \Cref{restrictionthm}(vi) and (vii), we have
$o(C)=q^2-1$, and since $C$ is a
subgroup of the multiplicative group of $\N$, $C=\cyc{\N^\star,\star}$ is
forced. By \Cref{restrictionthm} (ix), $r=o(C)/o(D)$ must divide $q+1$.
Since $o(D)|(q-1)$ by \Cref{restrictionthm} (viii), we have
$r\ge (q^2-1)/(q-1) = q+1$, and so $r=q+1$. This forces $D=\ffs{q}$, and
since $D$ is contained in the scalar field of $W$, $W=\ff{q}h$ now follows.
\end{proof}
We next wish to examine the form of $\delta$.
\begin{lem} \label{deltalem}
Let $U, G, H, W, C, D$ and $\delta$ be as in \Cref{restrictionthm}.
Assume $o(G)=q(q^2-1)$.
The following statements hold.
\begin{enumerate}[label=(\roman*)]
\item If $q\equiv 3\bmod 4$, then either
\begin{itemize}
\item $\delta(c)=c^{j(q+1)}$ for some $1\le j<q-1$ satisfying
$\gcd(j,(q-1)/2)=1$, or
\item the restriction of the homomorphism $\delta$ to
$c\in\square_{q^2}$ is given by
$\delta(c)=c^{j(q+1)/2}$ for some $1\le j<q-1$ satisfying
$\gcd(j,q-1)=1$.
\end{itemize}
\item If $q\equiv 1\bmod 4$, then
$\delta(c)=c^{j(q+1)}$ for some $1\le j<q-1$ satisfying
$\gcd(j,(q-1)/2)=1$.
\end{enumerate}
\end{lem}
\begin{proof}

By \Cref{CDWlem},
$C=\cyc{\N^\star,\star}=\square_{q^2}^\star\cup\square_{q^2}^\star \gen$, so
that the form of $\delta$ follows from \Cref{CDrelation} (ii).

First let us suppose that for $c\in\square_{q^2}^\star$, we have
$\delta(c)=c^{j(q+1)/2}$ for some $1\le j<q-1$ satisfying
$\gcd(j,q-1)=1$.
We know that $\delta:C\rightarrow D$ is a homomorphism.
Set $\gn=\gen^{q+1}$, which is a generator of $\ffs{q}$.
Now, $\delta(\gen)=\gn^i$ for some integer $1\le i<q-1$. We therefore have
\begin{equation*}
\delta(\gen\star\gen) =\delta(\gen)^2 =\gn^{2i}.
\end{equation*}
On the other hand $\gen\star\gen=\gen^{q+1}=\gn$, so that
\begin{equation*}
\delta(\gen\star\gen) =\delta(\gen^{q+1}) =\gn^{j(q+1)/2}.
\end{equation*}
We thus conclude that $2i\equiv j(q+1)/2\bmod (q-1)$.
However, if $q\equiv 1\bmod 4$, then $\gcd((q+1)/2,q-1)=1$, and so
$j(q+1)/2$ has an inverse mod $q-1$, whereas $2i$ does not. So,
$q\equiv3\bmod 4$ is forced in this case.

Now suppose that for $c\in\square_{q^2}^\star$, we have
$\delta(c)=c^{j(q+1)}$ for some $1\le j<(q-1)/2$ satisfying
$\gcd(j,(q-1)/2)=1$.
Proceeding as above, we find
$\gn^{2i}=\gn^{j(q+1)}=(\gn^{q+1})^j=(\gn^2)^j=\gn^{2j}$.
Thus, $2i\equiv 2j\bmod (q-1)$, or equivalently,
$i\equiv j\bmod ((q-1)/2)$.
Since $j<(q-1)/2$, we see $i=j$ or $i=j+(q-1)/2$.
If $i=j$, then there is nothing further to prove.
If $i=j+(q-1)/2$, then it is easily checked that for $c\in\square_{q^2}$,
$c^{j(q+1)}=c^{i(q+1)}$. We can therefore set $\delta(c)=c^{i(q+1)}$ again.
\Cref{restrictionthm} (iii) shows that in either case we can extend  the
definition of $\delta$ to include 0, and the result is proved.
\end{proof}
The above two lemmas now give us the form of the unital.
\begin{thm} \label{unitalform}
Let $U, G, H, W, C, D$ and $\delta$ be as in \Cref{restrictionthm}
and assume $o(G)=q(q^2-1)$.
Then
$$U=\{(x,b\, \delta(x)+w,1)\,:\, x\in\ff{q^2}\land w\in W\}\cup\{(0,1,0)\}$$
for some $b\in\ffs{q^2} \setminus W$.
\end{thm}
\begin{proof}
By \Cref{restrictionthm}, we already know $q+1$ points of $U$, namely,
$(0,1,0)$ and the set of points $\{(0,w,1)\,:\, w\in W\}$.
Recall that $[1,0,0]$ is the unique tangent line through $(0,1,0)$.
Consider the line $[1,0,-1]$, which also contains $(0,1,0)$ and is therefore
a secant line. Let $(1,b,1)$ be any one point of $U\cap [1,0,-1]$ with
$b\ne 0$. Note that $b\notin W$ by \Cref{restrictionthm}(iii).
The orbit of
this point under $G$ is
$\{(c,b\, \delta(c)+w,1)\,:\, c\in\ffs{q^2}\land w\in W\}$,
which gives us the remaining $q^3-q$ points of $U$. Since $\delta(0)=0$ in
either case of \Cref{deltalem}, we may describe $U$ as stated.
\end{proof}
We next establish a useful projective equivalence.
\begin{thm} \label{finalUform}
Let $U, G, H, W, C, D$ and $\delta$ be as in \Cref{restrictionthm}
and assume $o(G)=q(q^2-1)$.
\begin{enumerate}[label=(\roman*)]
\item If $\delta(x)=x^{j(q+1)}$,
then $U$ is projectively equivalent to a unital $\U(j)$ given by
$$\U(j) = \{(x,\Nm(x)^j+t\epsilon,1)\,:\, x\in\ff{q^2}\land t\in \ff{q}\}
\cup\{(0,1,0)\}.$$

\item If $\delta(x)=x^{j(q+1)/2}$ whenever $x\in\square_{q^2}$,
then $U$ is projectively equivalent to a unital that contains the
subset $\V(j)$ given by
\begin{equation*}
\V(j) =
\{(x,x^{j(q+1)/2}+t\epsilon,1)\,:\, x\in\square_{q^2}^\star\land t\in \ff{q}\}.
\end{equation*}
\end{enumerate}
\end{thm}
\begin{proof}
Set $d=\frac{h}{2} (b h^{-1} - b^q h^{-q})$, with $W=\ff{q} h$.
By \Cref{unitalform}, we know $b h^{-1}\notin\ff{q}$, and so $d\ne 0$.

First assume $\delta(x)=x^{j(q+1)}=\Nm(x)^j$, with
$1\le j<q-1$ and $\gcd(j,(q-1)/2)=1$.
We claim $\U(j)=U^\phi$, where $\phi=\phi(1,d,0,0)\in\Aut(\NP)$.
Suppose that $d\in\square_{q^2}$.
First, we note that
$(x,\Nm(x)^j+t\epsilon,1)^\phi= (x,(\Nm(x)^j+t\epsilon)d,1)$.
Thus, it is enough to show that for fixed $x\in\ff{q^2}$, we have
\begin{equation} \label{transformidentity}
\{b\, \Nm(x)^j+t'h\,:\, t'\in\ff{q}\}
=
\{(\Nm(x)^j+t\epsilon)d \,:\, t\in\ff{q}\}.
\end{equation}
Now, for fixed $t\in\ff{q}$ and $x\in\ff{q^2}$,
\begin{align*}
(\Nm(x)^j+t\epsilon)d
&=\frac{b}{2} \Nm(x)^j - \frac{h}{2} \Nm(x)^j b^qh^{-q} + td\epsilon\\
&=b\, \Nm(x)^j -\frac{h}{2}\Nm(x)^j
\left(bh^{-1} + b^qh^{-q}\right)+ td\epsilon\\
&=b\, \Nm(x)^j
+\left(td\epsilon h^{-1} -\frac{1}{2}\Nm(x)^j \Tr(bh^{-1})\right) h.
\end{align*}
Using the fact $\epsilon^q=-\epsilon$, it is easily verified that
$d\epsilon h^{-1}\in\ffs{q}$, and so we have
\begin{equation*}
b\Nm(x)^j+\left(td\epsilon h^{-1} -\frac{1}{2}\Nm(x)^j \Tr(bh^{-1})\right) h
 = b\Nm(x)^j+(t t_1 - t_2) h,
\end{equation*}
where $t_1,t_2\in\ff{q}$ are fixed.
Thus,
\begin{align*}
\{(\Nm(x)^j+t\epsilon)d \,:\, t\in\ff{q}\}
&=
\{(b\, \Nm(x)^j+ (t t_1 - t_2) h \,:\, t\in\ff{q}\}\\
&=
\{b\, \Nm(x)^j+t'h\,:\, t'\in\ff{q}\},
\end{align*}
which proves \Cref{transformidentity}. The case where
$d\in\nonsquare_{q^2}$ is very similar, the only difference being that we find
ourselves dealing with $\Nm(x)^jd - td\epsilon$, and the rest of the argument
follows in exactly the same way.

The proof for the case where $\delta(x)=x^{j(q+1)/2}$ is almost exactly the
same as that just given, differing only in the restriction of $\delta$ to
$x\in\square_{q^2}$. As such, we omit the details.
\end{proof}
In order to deal with the case with $\delta(x)=\Nm(x)^j$, we now define a
slightly more general form of $\U(j)$.
Specifically, for any $b\in\ffs{q}$, we set
\begin{equation*}
\U(b,j) = \{(x,b\, \Nm(x)^j+t\epsilon,1)\,:\, x\in\ff{q^2}\land t\in \ff{q}\}
\cup\{(0,1,0)\}.
\end{equation*}
We note that $\U(b,1)$ are known to be unitals thanks to the work of Wantz
\cite{W-2008-uitrn}.
\begin{lem}\label{lem-UBox}
Let $b,b'\in \ffs{q}$ be distinct.
The following statements hold.
\begin{enumerate}[label=(\roman*)]
\item $\U(b,j)=\{ (0,1,0)\} \cup \bigcup_{a\in \ff{q}} B(a,2 a^j b)$.
\item $\U(b,j)\cap \U(b',j)=\{ (0,1,0)\} \cup B(0,0)  $.
\item If $\U(b,j)$ is a unital, then any line $[u,1,z]$ with $u\ne 0$ that is
disjoint from $B(0,0)$ must be tangent to $\U(b,j)$ for exactly one
$b\in \ffs{q}$.
\item If $\U(b,j)$ is a unital, then any secant line $[u,1,z]$ with $u\ne 0$
that is disjoint from $B(0,0)$ must have at least 2 unital points whose
first coordinate lies in $\square_q^\star$, and at least 2 unital points whose
first coordinate lies in $\nonsquare_q$.
\item $B(0,0)\subseteq \V(j)$ and for $c,d\in\ffs{q}$, we have
\begin{equation*}
|\V(j) \cap B(c,d)| =
\begin{cases}
\frac{q+1}{2} &\text{ if $c=a^2$ and $d=2a^j$ for some $a\in\ffs{q}$,}\\
\ \ 0 &\text{ otherwise.}
\end{cases}
\end{equation*}
\end{enumerate}
\end{lem}
\begin{proof}
For (i), we have
\begin{align*}
\U(b,j)
&=\{ (0,1,0)\} \cup
\{(x , b\, \Nm(x)^j + t\epsilon,1)\,:\, x \in \ff{q^2}, t \in \ff{q}\} \\
&=\{ (0,1,0)\} \cup
\bigcup_{a\in\ff{q}} \{(x , ba^{j} + t\epsilon ,1)\,:\,\Nm(x)=a, t\in\ff{q}\}\\
&= \{ (0,1,0)\} \cup \displaystyle \bigcup_{a\in \ff{q}} B(a,2 a^jb),
\end{align*}
where the last step follows from
$\Tr(ba^j+t\epsilon)=\Tr(a^j b)+t\Tr(\epsilon)=2a^ jb+0$.

If $(x,y,1)\in B(a,2  a^j b)\cap B(a,2 a^j b')$, then
$\Tr(y) = 2 a^j  b = 2 a^j b'$, which implies $a=0$ or $b=b'$. Applying (i) now
proves (ii).

For (iii), suppose $\lne{l}=[u,1,z]$ with $u\ne 0$ is disjoint from $B(0,0)$.
Let $n_s,n_t$ be the numbers of $b$'s in $\ffs{q}$ such that $\lne{l}$ is secant
to and tangent to $\U(b,j) \setminus  (B(0,0) \cup \{ (0,1,0)\} )$,
respectively. Then $n_s+n_t=q-1$.
On the other hand, the set of affine points is the disjoint union
\begin{align*}
\displaystyle \bigcup_{a,b\in \ff{q}} B(a,b)
&=S\cup \left(\bigcup_{a,b\in \ff{q}^\star} B(a,b) \right)\\
&=S\cup \left(\bigcup_{a,b\in \ff{q}^\star} B(a,2 a^j b) \right)\\
&=S\cup \left(\bigcup_{b\in \ff{q}^\star} \U(b,j) \setminus (B(0,0)
\cup \{ (0,1,0)\} ) \right),
\end{align*}
where $S=\bigcup_{ab=0} B(a,b)$ is the set of
points on $[1,0,0]$ or on the $q$ horizonal lines $[0,1,c]$ with $\Tr(c)=0$.
So $\lne{l}$ intersects $S$ in exactly $q+1$ points.
Counting the number of affine points of $\lne{l}$ gives
$(q+1)n_s + n_t + (q+1) = q^2$.
Together with $n_s + n_t = q -1$, on solving, we have $n_s = q-2$ and $n_t=1$, proving (iii).

If we now further assume $\lne{l}$ is a secant, then because of a property of
the norm, when $a$ runs over $\ffs{q}$, precisely $\frac{q-1}{2}$ of
$B(a,2 b a^j)$ are composed of points whose first coordinates are squares.
So, by \Cref{replacementlemma3}
(iv), there are at most $2(\frac{q-1}{2}) = q-1$ unital points on
$\lne{l}$ with first coordinates in $\square_q^\star$. A similar argument
works for the non-square case, which proves (iv).

For (v), clearly $B(0,0)\subseteq \V(j)$.
For $(x,\delta(x)+t\epsilon,1)\in \V(j)$ with $x\in\square_{q^2}^\star$,
we have $\delta(x)=a\in\ffs{q}$.
Thus, $\Tr(\delta(x)+t\epsilon)=2a$. Now,
$$\Nm(x)^{j} = x^{j(q+1)}=\delta(x)^2=a^2,$$
so that $\Nm(x)=a^{2j'}$. Thus, $(x,\delta(x)+t\epsilon,1)\in B(a^{2j'},2a)$
for some $a\in\ffs{q}$. It remains to determine the cardinality of
$\V(j)\cap B(a^{2j'},2a)$.
Set $R=\{\gen^{i(q-1)}\,:\, 0\le i\le q\}$. Then $R$ is precisely
the elements of norm 1 in $\ff{q^2}$. Furthermore, $R\subseteq\square_{q^2}$.
If $\delta(x)=a\in\ffs{q}$, then we observe
\begin{align*}
\Tr(\delta(\gen^{i(q-1)}x) + t\epsilon)
&= \Tr(\delta(\gen^{i(q-1)}x))\\
&= \Tr(\gen^{ij(q^2-1)/2} a)\\
&= \Tr((-1)^i a)\\
&=(-1)^i 2a,
\end{align*}
which proves the $q+1$ elements of norm $a^{2j'}$ split evenly among the
elements of trace $\pm 2a$. Finally, we may set $a'^j=a$ so that
$B(a^{2j'},2a)=B(a'^2,2a'^j)$, completing the proof.
\end{proof}

We are finally in a position to prove the following classification result.
\begin{thm} \label{TheWantzTheorem}
Let $U$ be a parabolic unital of order $q$ containing $(0,1,0)$ or $(1,0,0)$ in the regular
nearfield plane $\N=\N(2,q^2)$ with the largest possible automorphism group.
Then $o(\Aut(U))=q(q^2-1)$ and $U$ is projectively equivalent to
$$\{(x,\Nm(x)+t\epsilon\,:\, x\in\ff{q^2}, t\in\ff{q}\} \cup \{(0,1,0)\},$$
where $\epsilon\in\ffs{q^2}$ satisfies $\epsilon^q=-\epsilon$.
In particular, $U$ is projectively equivalent to a Wantz unital.
\end{thm}
\begin{proof}
Thanks to the work above, we know that $U$ is either projectively equivalent
to $\U(j)$ or to a unital containing $\V(j)$. We therefore have two tasks.
The first, and more difficult, task is to show $\U(j)$ is
not a unital for any $1<j<q-1$ with $\gcd(j,q-1)=1$.
This will prove only $\U(1)$ is a unital among the sets $\U(j)$.
The second task is to show no unital can contain $\V(j)$.

To begin, say $1\le j\le (q-1)/2$, and suppose both $\U(j)$ and $\U(j+(q-1)/2)$
are unitals.
Then $\U(j+(q-1)/2)^{\phi(1,-1,0,0)}=\U(-1,j + (q-1)/2)$ is also a
unital. Let $\lne{l}=[1,1,u]$ be a tangent to $\U(j)$.
Observe that $(x,y,1)\in \lne{l}\cap \U(b,j)$ if and only if
$x + y + u =0$ and $\Tr(y)=2 b \Nm(x)^j$.
Applying this observation to the three unitals $\U(j)$, $\U(j+(q-1)/2)$,
and $\U(-1,j+(q-1)/2)$, we get that a point $(x,y,1)$ on $\lne{l}$ is on each
such unital precisely when the following hold, respectively:
\begin{align}
\label{eq:1}
\Tr(-x-u) &= 2\Nm(x)^{j}, \\
\label{eq:2}
\Tr(-x-u) &= 2\Nm(x)^{j + \frac{q-1}{2}}, \\
\label{eq:3}
\Tr(-x-u) &= - 2\Nm(x)^{j + \frac{q-1}{2}}.
\end{align}
Since $\Nm(x)^{\frac{q-1}{2}}=\pm 1$, depending on whether $x\in\square_{q^2}$,
the square solutions of \eqref{eq:1} and \eqref{eq:2} are
identical, as are the nonsquare solutions of \eqref{eq:1} and \eqref{eq:3}.
Since $\lne{l}$ is tangent to $\U(j)$, there is only one solution to
\eqref{eq:1}. So either \eqref{eq:2} has exactly one square solution or
\eqref{eq:3} has exactly one non-square solution.
Either case contradicts \Cref{lem-UBox}(iv).
We conclude that at most one of $\U(j)$ and $\U(j+(q-1)/2)$ can be a unital.
Since $\U(1)$ is known to be a unital, we find $\U((q+1)/2)$ is not.

Now suppose $\U(j)$ is a unital with $1<j<q-1$, $\gcd(j,(q-1)/2)=1$, and
$j \ne(q+1)/2$.
We next claim $\gcd(2j-1,q-1)=1$. Suppose otherwise.
Then $x^{2j-1} = 1$ has at least two solutions, say $c_1,c_2$, over $\ff{q}$.
Let $\pnt{P}=(0,0,1)$, $\pnt{P}_{i1}=(c_i,c_i,1)$ and
$\pnt{P}_{i2}=(c_i,c_i+c_i \epsilon,1)$ for $i=1,2$.
Firstly, it is easily checked that the six points
$\pnt{Q}=(0,1,0)$, $\pnt{P}$ and $\pnt{P}_{ii'}$, $i,i'=1,2$,
form an O'Nan configuration. Indeed, one checks that
\begin{align*}
\{\pnt{P},\pnt{P}_{11},\pnt{P}_{21}\}&\subseteq [-1,1,0].\\
\{\pnt{P},\pnt{P}_{12},\pnt{P}_{22}\}&\subseteq [-1-\epsilon,1,0],\\
\{\pnt{Q},\pnt{P}_{11},\pnt{P}_{12}\}&\subseteq [1,0,-c_1], \text{ and}\\
\{\pnt{Q},\pnt{P}_{21},\pnt{P}_{22}\}&\subseteq [1,0,-c_2]
.
\end{align*}
Note that none of these lines are of the form $[0,1,z]$.
Since each $c_i \in \ff{q}$ is a solution of $x^{2j-1} = 1$, we have
$\Nm(c_i)^j={c_i}^{(q-1)j}{c_i}^{2j-1}c_i=c_i$.
Thus, the six points are in $\U(j)$ and hence $\U(j)$ is not a unital by
\Cref{thmONan}, a contradiction.
So $\gcd(2j-1,q-1)=1$.

For $U$ to be a unital, then, we have $1<j<q-1$,
$\gcd(j,(q-1)/2)=1=\gcd(2j-1,q-1)$, and $j \ne(q+1)/2$.
Set $j=j'+t(q-1)/2$, with $1<j'<(q-1)/2$ and $t\in\{0,1\}$.
Note that $\gcd(2j'-1,q-1)=1$.
Now, with $E=\ff{q}\setminus\{0,1\}$, if $x\in E$, then
\begin{align*}
h_{2j-1}(x) &= \frac{x^{2j}-1}{x-1}\\
&= \frac{x^{2j'+q-1}-1}{x-1}\\
&= \frac{x^{2j'}-1}{x-1}\\
&= h_{2j'-1}(x).
\end{align*}
Since $\gcd(2j'-1,q-1)=1$, we may apply \Cref{allonespoly} to obtain distinct
elements $c_1',c_2' \in E$ satisfying
$\frac{c_1'^{2j}-1}{c_1'-1}=\frac{c_2'^{2j}-1}{c_2'-1}$.
Let $\pnt{P}'=(1,1,1)$, $\pnt{P}_{i1}'=(c_i',c_i'^{2j},1)$,
$\pnt{P}_{12}'=(c_1',c_1'^{2j}+\epsilon,1)$,
$\pnt{P}_{22}'=(c_2',c_2'^{2j}+t \epsilon,1)$ for $i=1,2$,
where $t=\frac{c_2'-1}{c_1'-1}\in \ff{q}$.
With repeat appeals to \Cref{lem-collinear} we find the six points
$\pnt{Q}$, $\pnt{P}'$ and $\pnt{P}_{i i'}'$, where $i ,i'=1,2$, form an O'Nan
configuration with no line of the form $[0,1,z]$. It is easily seen all
points lie in $\U(j)$, so that we again have a contradiction of \Cref{thmONan}.
This completes our first task, and we now know $\U(1)$ is the only unital among
the sets $\U(j)$.

Now suppose $U$ is projectively equivalent to a unital containing $\V(j)$.
This case can only arise if $q\equiv 3\bmod 4$. We first prove
$\gcd(j(q+1)/2-1,q-1) = 1$. If it is not, then there exists elements
$d_1,d_2\in\ffs{q}\subseteq\square_{q^2}^\star$ that satisfy
$\delta(d_i)=d_i$. Subsequently, the points
$\pnt{R}_{i1}=(d_i,d_i,1)$ and $\pnt{R}_{i2}=(d_i,d_i+d_i \epsilon,1)$, for
$i=1,2$, lie in $\V(j)$.
One verifies these 4 points, along with $\pnt{P, Q}$, form an O'Nan
configuration in exactly the same way as done for the points
$\pnt{P,Q}, \pnt{P}_{i1}, \pnt{P}_{i2}$. This yields a contradiction, proving
$\gcd(j(q+1)/2-1,q-1) = 1$.

Next, we note that for $x\in\ffs{q}$, we have
$$\delta(x) = x^jx^{j(q-1)/2} = x^j \eta(x)^j = x^j \eta(x) = x^{j+(q-1)/2},$$
where $\eta(x)$ is the quadratic character over $\ffs{q}$.
If $1\le j< (q-1)/2$, then we set $k=j+(q-1)/2$, which yields
$\delta(x)=x^k$ with $(q-1)/2 < k < q-1$.
Otherwise $(q-1)/2< j < q-1$, and we instead set $k=j-(q-1)/2$, so that
$\delta(x)=x^{q-1}x^k=x^k$ with $1\le k < (q-1)/2$.
Since $x^n=x^{n\bmod (q-1)}$ for any $x\in\ffs{q}$, we thus also conclude
$j(q+1)/2\equiv k\bmod (q-1)$. In particular, since we know
$\gcd(j(q+1)/2-1,q-1) = 1$ we also obtain $\gcd(k-1,q-1)=1$.
We may therefore appeal to \Cref{allonespoly} to obtain elements
$d_1',d_2' \in E$ satisfying
$\frac{\delta(d_1')-1}{d_1'-1}=\frac{\delta(d_2')-1}{d_2'-1}$.
Let $\pnt{R}_{i1}'=(d_i',\delta(d_i'),1)$,
$\pnt{R}_{12}'=(d_1',\delta(d_1')+\epsilon,1)$,
$\pnt{R}_{22}'=(d_2',\delta(d_2')+t \epsilon,1)$ for $i=1,2$,
where $t=\frac{d_2'-1}{d_1'-1}\in \ff{q}$.
Again we appeal to \Cref{lem-collinear} several times to verify that the six
points $\pnt{Q}$, $\pnt{P}'$ and $\pnt{R}_{i i'}'$, where $i,i'=1,2$, form an
O'Nan configuration with no line of the form $[0,1,z]$. It is easily seen all
points lie in $\V(j)$, so that we again have a contradiction of \Cref{thmONan}.
Thus there is no unital containing $\V(j)$.

\end{proof}

Any orthogonal-BM unital in the Desarguesian plane $\PG(2,q^2)$ contains no O'Nan configurations through the unital point at the line at infinity \cite{BE-1992-obmuo}. We finish the paper by extending this result to the Wantz unitals in the nearfield plane.

\begin{lem}
The Wantz unitals contain no O'Nan configurations through the unital point at the line at infinity.
\end{lem}

\begin{proof}
It suffices to show that the statement is true for $\U(1)$ whose unique 
unital point at the line at infinity is $(0,1,0)$.
By \Cref{thmONan}, one of the lines in the configuration must be $[0,1,z]$ for some $z\in\N$.
Applying \Cref{replacementlemma3} (i), (ii), (iii) to the form of $\U(1)$ given
in \Cref{lem-UBox}(i), we deduce that the remaining 5 points of the
configuration are in $B(a,2 a b )$ for a fixed ${a,b\in\ffs{q}}$.
The configuration must contain at least one line of the form
$[u,1,z']$ with $u\ne 0$, and this line must contain 3 points of $\U(1)$.
However, it follows from \Cref{replacementlemma3} (iv) that $[u,1,z']$ can
only contain 2 such points. Contradiction.
\end{proof}

\bibliographystyle{amsplain}

\providecommand{\bysame}{\leavevmode\hbox to3em{\hrulefill}\thinspace}
\providecommand{\MR}{\relax\ifhmode\unskip\space\fi MR }
\providecommand{\MRhref}[2]{%
  \href{http://www.ams.org/mathscinet-getitem?mr=#1}{#2}
}
\providecommand{\href}[2]{#2}

\end{document}